\documentclass[12pt]{amsart}
\usepackage[utf8]{inputenc}
\usepackage{amsfonts,amsmath,amssymb,amsxtra,stmaryrd,mathdots, amsthm}

\usepackage{a4wide}

\usepackage{colonequals}
\usepackage[utf8]{inputenc}
\usepackage[british]{babel}
\usepackage{hyperref}
\hypersetup{colorlinks=true,urlcolor=blue,citecolor=blue,linkcolor=blue}

\usepackage[all,cmtip]{xy}
\usepackage{float}

\usepackage[dvipsnames]{xcolor}
\usepackage[normalem]{ulem}
\usepackage{cleveref}
\Crefname{equation}{}{}

\usepackage{tikz}
\usepackage{tikz-cd}

\usepackage{mathtools}
\usepackage{amsmath}
\usepackage{thmtools}
\usepackage{thm-restate}

\usepackage{tikz-cd}
\usetikzlibrary{patterns,shapes,cd,arrows}
\usepackage[auto]{contour}
\contourlength{2pt}

\usepackage{makecell}

\usepackage{thm-restate}

\usepackage{listings}
\usepackage{xcolor}

\definecolor{codegreen}{rgb}{0,0.6,0}
\definecolor{codegray}{rgb}{0.5,0.5,0.5}
\definecolor{codepurple}{rgb}{0.58,0,0.82}
\definecolor{backcolour}{rgb}{0.95,0.95,0.92}

\lstdefinestyle{mystyle}{
    backgroundcolor=\color{backcolour},   
    commentstyle=\color{codegreen},
    keywordstyle=\color{magenta},
    numberstyle=\tiny\color{codegray},
    stringstyle=\color{codepurple},
    basicstyle=\ttfamily\footnotesize,
    breakatwhitespace=false,         
    breaklines=true,                 
    captionpos=b,                    
    keepspaces=true,                 
    numbers=left,                    
    numbersep=5pt,                  
    showspaces=false,                
    showstringspaces=false,
    showtabs=false,                  
    tabsize=2
}

\usepackage{enumitem}%

\newcommand{\Q}{\mathbb{Q}}

\newcommand{\C}{\mathbb{C}}

\newcommand{\Z}{\mathbb{Z}}

\DeclareMathOperator{\End}{\operatorname{End}}

\DeclareMathOperator{\GL}{GL}

\DeclareMathOperator{\Hom}{Hom}
\DeclareMathOperator{\Gal}{Gal}

\DeclareMathOperator{\MT}{MT}

\DeclareMathOperator{\Spec}{Spec}

\DeclareMathOperator{\Ind}{Ind}
\DeclareMathOperator{\Res}{Res}
\DeclareMathOperator{\HH}{H}

\DeclareMathOperator{\proj}{proj}

\newcommand{\dR}{\mathrm{dR}}

\usepackage{mathrsfs}

\numberwithin{equation}{section}

\newtheorem{theorem}[equation]{Theorem}

\newtheorem{corollary}[equation]{Corollary}

\newtheorem{lemma}[equation]{Lemma}

\newtheorem{proposition}[equation]{Proposition}

\theoremstyle{definition}
\newtheorem{definition}[equation]{Definition}

\theoremstyle{remark}
\newtheorem{remark}[equation]{Remark}

\newtheorem{example}[equation]{Example}

\newcommand{\Gl}{\mathcal{G}_\ell}

\DeclareMathOperator{\etale}{\text{ét}}

\DeclareMathOperator{\AH}{AH}

\newcommand{\defeq}{\vcentcolon=}

\usepackage{wrapfig} \title[Connected monodromy fields of Jacobians with complex multiplication]{Connected monodromy fields of Jacobians\\with complex multiplication}

\author{Andrea Gallese and Davide Lombardo}

\begin{document}

\begin{abstract}
    We describe an algorithm to compute the minimal field of definition of the Tate classes on powers of a Jacobian $J$ with potential complex multiplication. This field arises as a natural invariant of the Galois representations attached to $J$. We also give closed formulas expressing the periods of anti-holomorphic differential forms on $J$ in terms of the periods of the holomorphic ones.
\end{abstract}

\maketitle

\section{Introduction}
Let $A$ be an abelian variety defined over a number field $k$. 
The \emph{Galois representation} attached to $A$ relative to a prime $\ell$ is the homomorphism
\[
    \rho_{A,\ell} \colon \Gal(\bar k /k )\to \GL(V_\ell)
\]
arising from the natural Galois action on the Tate module $V_\ell = \varprojlim_n A[\ell^n] \otimes_{\Z_\ell}\Q_\ell$.
Galois representations of this form play a central role in arithmetic geometry, and many deep results in number theory have been obtained by reducing questions to the study of such representations; see, for example, the surveys~\cite{MR2060030, MR4871889}. %

The \emph{$\ell$-adic monodromy group} of $A$, denoted by $\Gl$, is defined as the Zariski closure inside $\GL_{V_\ell}$ of the image of $\rho_{A,\ell}$. It is a reductive linear algebraic group over $\mathbb{Q}_\ell$, and a problem of fundamental interest is to determine its isomorphism class. In this paper, we solve this problem under the assumption that $A$ has complex multiplication.

This problem naturally decomposes into three tasks: describing the identity component $\Gl^0$, understanding the finite component group $\Gl/\Gl^0$, and describing the class of their extension. Most of our work concerns the determination of the component group. More generally, we study the homomorphism
\[
    \varepsilon_{A, \ell}\colon \Gal(\bar{k}/k) \longrightarrow \Gl(\mathbb{Q}_\ell) \longrightarrow (\Gl/\Gl^0)(\mathbb{Q}_\ell).
\]
The component group is finite, and each of its elements admits a $\mathbb{Q}_\ell$-rational representative. The projection $\varepsilon_{A,\ell}$ is surjective, and therefore identifies $\Gl/\Gl^0$ with the Galois group of a finite extension $k(\varepsilon_A)/k$. A theorem of Serre~\cite[n.~133]{serre-IV} shows that this extension is independent of $\ell$. We refer to $k(\varepsilon_A)$ as the \emph{connected monodromy field} of $A$.  
The connected monodromy field admits several equivalent characterizations:
\begin{itemize}
    \item It is the smallest extension $L/k$ such that the monodromy group of $A_L$ is connected~\cite[Proposition 2.3]{zywina2019effective}.
    \item It is the field of definition of all Tate classes on powers $A^r$~\cite[Proposition 2.4.3]{cantoralfarfan2023monodromy}.
    \item It %
    is the intersection of the $\ell^\infty$-torsion fields of $A$~\cite[Theorem 0.1]{MR1441234}:
    \begin{equation}\label{eq: characterization 3}
        \textstyle k(\varepsilon_A) = \bigcap_{\ell \text{ prime}}k(A[\ell^\infty]).
    \end{equation}
\end{itemize}

For many abelian varieties, the connected monodromy field coincides with the field of definition of the endomorphisms of $A$ \cite{MR3320526}, but at present, there is no general algorithm to compute the connected monodromy field of an arbitrary abelian variety. Existing approaches are typically ad hoc. In some cases, the field can be determined only after extending the base field $k$~\cite{MR1630512}; in others, one restricts to curves with explicitly described algebraic cycles~\cite{cantoralfarfan2023monodromy}, or to special families such as Fermat Jacobians~\cite{gallese_part1}. By contrast, in this paper we present an algorithm that computes the connected monodromy field for any CM Jacobian.

\subsection*{Main result}
When $A$ has complex multiplication, an explicit description of the connected component of $\Gl$ is available. Indeed, the Mumford--Tate conjecture is known to hold for $A$~\cite{MTconjMC}, implying that $\Gl^0 \simeq \MT(A)_{\mathbb{Q}_\ell}$, and the Mumford--Tate group can be computed explicitly~\cite{LombardoMT}. Our main result, \Cref{theorem: computation of kconn} below, gives an additional characterization of the connected monodromy field.  

Let $E$ denote a maximal CM subalgebra of $\End(A_{\bar k})$, and let $F/\mathbb{Q}$ be the normal closure of the compositum of the components of $E$.  
Exploiting the theory of complex multiplication, we compute algebraic differential $1$-forms $\omega_1,\dots,\omega_{2g}$ whose classes form a basis of $\HH^1_{\dR}(A/kF)$ and such that:
\begin{itemize}
    \item $\omega_1, \, \dots, \,\omega_g$ are regular differential forms giving a basis of $\HH^0(A_{kF}, \Omega^1)$;
    \item $\omega_1,\, \ldots,\, \omega_{2g}$ form an eigenbasis for the $E$-action: there are $2g$ characters $\chi_i\colon E \to \C$ with
    $\alpha^\ast(\omega_i) = \chi_i(\alpha)\cdot\omega_i$, 
    as we explain in \Cref{section: compute the CM type};
    \item the $\Gal(F/k)$-action on $\omega_1,\, \ldots,\, \omega_{2g}$ is a permutation representation, that is, for each $\tau \in \Gal(F/k)$ there is a permutation of $\{1, \ldots, 2g\}$, that we still denote by $\tau$, such that
    $\tau(\omega_i)= \omega_{\tau(i)}$ for all $i$ (see \Cref{lemma: rescaling of the eigenbasis}). The Galois action on the index set $\{1, \dots, 2g\}$ corresponds to its action on the characters $\{\chi_1, \dots, \chi_{2g}\}$, so that $\omega_{\tau (i)}$ is the unique $\tau\chi_i$-eigenform.
\end{itemize}

In this basis, the Mumford--Tate group $\MT(A)$ is diagonal. Writing $x_i$ for the diagonal entries, one may find a finite generating set $\mathcal{F}$ of equations for $\MT(A)$ of the form $f=1$, where $f=\prod_i x_i^{e_i}$ is a Laurent monomial. In \Cref{definition: differential forms line decomposition}, we associate with any $f\in \mathcal{F}$ with $\sum_{i} |e_i| = 2n$ a differential form $\omega_f \in \HH_{\dR}^1(A/kF)^{\otimes 2n}$ which is an eigenvector for the action of $E^{\otimes 2n}$ with character $\bigotimes_i \chi_i^{\otimes e_i}$.
Let $\tau f$ denote the Laurent monomial $\prod_i x_{\tau(i)}^{e_i}$.

We now show how to characterize the connected monodromy field in terms of \textit{periods}, which we now introduce.
\begin{definition}
    \label{definition: periods}
    
    For an abelian variety $A$, a cycle $\sigma\in \HH_{2n}^B(A(\C), \Q)$, and a differential $2n$-form $\omega$ on $A(\C)$, define the period
    \[
        P(\sigma, \omega) \defeq \frac{1}{(2\pi i)^n} \int_\sigma \omega.
    \]
\end{definition}

Periods naturally arise when comparing the singular and de Rham cohomology of a variety. Grothendieck’s period conjecture \cite[Section 7.5]{MR2115000} predicts that, in general, the algebraic relations among such integrals are precisely those explained by the motivic structure of the variety. For CM abelian varieties, Deligne \cite{Deligne} showed the algebraicity of suitable periods, making them particularly amenable to computation, and it is this description that we exploit in our work. %
More precisely, fix an equation $f \in \mathcal{F}$ and a topological cycle $\lambda$ that generates $\HH_1^B(A, \Q)$ as a free $E$-module. Let $\sigma \in \HH^B_{2n}(A^{2n}(\C), \Q)$ be the class corresponding to $\lambda^{\otimes 2n} $ via the {K\"unneth} product formula.
In \Cref{proposition: CI is Galois stable iff period condition}, we prove that the period $P(\sigma, \omega_f)$ is algebraic. The periods $P(\sigma, \omega_f)$ for $f$ ranging over a set of defining equations of $\MT(A)$ can be used to describe $k(\varepsilon_A)$, as made precise in the following statement.

\begin{theorem}
{\label{theorem: computation of kconn}}
    Let $A/k$ be an abelian variety with complex multiplication.
    Let $\mathcal{F}$ be a finite generating set of equations for $\MT(A)$.
    The extension $k(\varepsilon_A)/k(\End A)$ is abelian and corresponds to the subgroup of $\tau \in \Gal(\bar k / k(\End A))$ such that
    \[ \tau P(\sigma, \omega_f) = P(\sigma, \omega_{\tau f}) \quad \forall f \in \mathcal{F}. \]
\end{theorem}
\begin{proof}
    The inclusion $k(\varepsilon_A) \supseteq k(\End A)$ holds for any abelian variety \cite[Proposition 2.10]{SilvZarhin}. Characterization
    \Cref{eq: characterization 3} yields
    \[
        \textstyle k(\varepsilon_A) = \bigcap_\ell k(\End(A), A[\ell^\infty])
    \]
    and the extensions $k(\End(A), A[\ell^\infty])/k(\End A)$ are abelian by CM theory \cite[Proposition 7.3.(a)]{MilneCM}.
    Fix an auxiliary prime $\ell$ %
    and consider the $\ell$-adic representation $\rho_\ell$.
    By \cite[Theorem 2.2.2]{gallese_part2}, the image $\rho_\ell(\tau)$ lies in the connected component of the monodromy group if and only if it fixes pointwise the algebra of Tate classes $W_\ell$.

    The key insight is that $W_\ell$ is the $\ell$-adic étale realization of the algebra of absolute Hodge cycles $C_{\AH}^\vee$, as a Galois representation: see \Cref{proposition: Well is the ladic realization of CAH}. %
    In turn, the Galois action on the algebra of absolute Hodge cycles can be computed on its de Rham realization.
    It follows from \Cref{corollary: eigenspace decomposition of CAH} and \Cref{proposition: CI is Galois stable iff period condition} %
    that the action is trivial exactly when the period condition in the statement is satisfied. %
\end{proof}

\begin{remark}
    The arguments above can be pushed further to describe the full monodromy group $\Gl$. Remarkably, the resulting description is independent of the auxiliary prime $\ell$.
    Let $V = \HH_1^B(A(\C), \Q)$. The isomorphism $W \simeq C_{\AH}^\vee$ of \Cref{prop: W is the etale realisation of CAH} allows $\GL_V$ to act on the algebra of absolute Hodge classes.
    Define the algebraic subgroup $\mathcal{G}\subseteq \GL_V$ such that: for all $\Q$-algebras $L$ we have 
    \[ \mathcal{G}(L) =  \{ g \in \GL_V(L) \mid \exists \tau \in \Gal(k(\varepsilon_A)/k) \text{ s.t. } g(w) = \tau(w) \quad \forall w \in C_{\AH}^\vee \otimes_\Q L \}.  \]
    Arguing as in the proof of \Cref{theorem: computation of kconn} and applying \cite[Theorem 2.2.2]{gallese_part2} to describe $\Gl$, it follows that $\mathcal{G}\times_{\Spec\Q} \Spec\Q_\ell \simeq \Gl$ for all primes $\ell$. One can then compute the Galois action on $C_{\AH}^\vee$ via its de Rham realization as in the proof of \Cref{proposition: CI is Galois stable iff period condition} and derive explicit equations for $\mathcal{G}$. %
    See also \Cref{remark: Galois action without assumption 2}.
\end{remark}

\subsection*{Computations}
If $A$ is the CM Jacobian of a smooth projective curve, one can numerically compute the periods appearing in \Cref{theorem: computation of kconn} and identify them with algebraic numbers. From this identification, one obtains a candidate minimal polynomial for one of the generators of $k(\varepsilon_A)$. To turn this computation into a formal proof, we would need to certify that the recognized minimal polynomial is indeed correct. This is sometimes possible, as shown for example in \cite[Section 6.3]{gallese_part1}, where the periods are computed as exact algebraic numbers in the case of the Fermat curves.

\begin{remark}
    The method applies to any CM abelian variety $A$, provided one has access to the period matrix in the eigenbasis $\{\omega_i\}$. In the remainder of the paper we focus on the case of Jacobians, detailing the computations in this setting. In \Cref{remark: general case n1}, \Cref{remark: general case n2}, and \Cref{remark: general case n3}, we isolate the hypotheses on $A$ required for the generalizations.
\end{remark}

In \Cref{section: examples} we illustrate the algorithm with explicit computations. In particular, we determine the connected monodromy field of the Jacobian $J$ of a curve $X$ in three cases where $J$ is degenerate (that is, its Hodge ring is not generated by divisors). The corresponding \textsc{Magma} code for these computations is available on GitHub \cite{myrepo2025}.

\subsection*{Notation}
For ease of reference, we collect here the main notations used in the rest of the paper.
    \begin{itemize}
    \item $J$ is a $g$-dimensional abelian variety over the number field $k$, with complex multiplication by the CM algebra $E = E_1 \times \cdots \times E_t$.
    \item $\Phi= \{\chi_1, \,\ldots,\, \chi_g\}$ is the CM type of $J$. We set $\chi_{i+g}=\bar{\chi}_i$ for $i=1,\ldots,g$.
    \item $F$ is a finite Galois extension of $\Q$ containing each $E_i$. 
    \item $V$ is the first singular cohomology group $\HH_1^B(A(\C), \Q)$. We set $V_F \defeq V \otimes_\Q F$.
    \item $\omega_1,\, \ldots,\, \omega_{g}$ is a basis of $\HH^0(J_{kF}, \Omega^1)$ such that $\alpha \cdot \omega_i=\chi_i(\alpha) \omega_i$ for $i=1,\ldots,g$ and all $\alpha \in E$. We extend this to an eigen-basis $\omega_1, \, \dots, \, \omega_{2g}$ of $\HH^1_{\dR}(J/k)$ for the $E$-action.
    
    \item $v_1,\, \ldots,\, v_{2g}$ is a basis of $V_F$ such that $\alpha \cdot v_i = \chi_i(\alpha) v_i$ for all $i=1,\ldots, 2g$.
\end{itemize}

\subsection*{Acknowledgements} We thank Thomas Bouchet, Jeroen Hanselman, Andreas Pieper, and Sam Schiavone for sharing with us the equation of the curves studied in \Cref{ex: Mumford I,ex: Mumford II}.
D.L. was supported by the University of Pisa through grant PRA-2022-10 and by MUR grant PRIN-2022HPSNCR (funded by the European Union project Next Generation EU). Both authors are members of the INdAM group GNSAGA.

\section{Computing the endomorphism ring}
Let $X$ be a smooth, proper, geometrically connected curve of genus $g$ defined over a number field $k\subseteq \C$.
Denote by $J$ its Jacobian.
We can compute the geometric endomorphism ring $\End(J_{\Bar{k}})$ using the algorithm of \cite{CMSV, CostaLombardoVoight}. 

Fix a basis $\omega_1, \, \dots, \, \omega_g$ for the $k$-vector space of Kähler differentials $\HH^0(X, \Omega_X^1)$ and a basis for the first singular cohomology group $\HH^1_B(J(\C), \Z)$, symplectic with respect to the intersection form. Consider the complex uniformization $J({\C}) \simeq \C^g/\Lambda$ and denote by $\Pi \in M_{g \times 2g}(\C)$ the period matrix in the fixed bases (that is, $\Pi_{ij}=\int_{\gamma_j} \omega_i$ for $i=1,\ldots,g$ and $j=1,\ldots,2g$). Recall that there is a natural identification $\Lambda = \HH_1^B(J(\C), \Z)$.

The \emph{tangent representation} of the endomorphism algebra is defined by left multiplication on the tangent space $\HH^0(X, \Omega_X^1)^\vee\otimes_k \C$. An endomorphism $\alpha$ corresponds to a linear transformation that respects the lattice $\Lambda \subseteq \C^g$. Thus, in the fixed basis, $\alpha$ is represented by a complex matrix $M_\alpha \in M_{g \times g}(\C)$ such that
\begin{equation}
{\label{eq: Period matrix endomorphism relation}}
    M_\alpha \Pi = \Pi A_\alpha
\end{equation}
for a suitable integral matrix $A_\alpha \in M_{2g \times 2g}(\Z)$. The coefficients of the matrix $M_\alpha$ are algebraic -- that is, $M_\alpha$ lies in $M_{g \times g}(\bar{\mathbb{Q}})$ -- so we can (usually) recognize $M_\alpha$ numerically as a matrix with algebraic coefficients, see \cite[Section 2.2]{CMSV}.
The \emph{rational representation} $V = \Lambda \otimes_\Z \Q$ is the map $\alpha\mapsto A_\alpha$.
This allows for a day-and-night algorithm that during the day computes pairs $(M,A)$ satisfying relation \Cref{eq: Period matrix endomorphism relation} (up to some numerical precision), and during the night certifies whether they correspond to an endomorphism. In \cite{CostaLombardoVoight}, this algorithm is proven to terminate under the assumption of the Mumford-Tate conjecture for $J$. %

\begin{lemma}
    The field of definition of the endomorphisms $k(\End(J))$ is the field over which the coefficients of all $M_i$ are defined.
\end{lemma}
\begin{proof}
    This follows from a stronger statement: the field of definition of the endomorphism $\alpha \in \End(J_{\bar k})$ is the field generated by the coefficients of the corresponding matrix $M_\alpha$.
The tangent representation, which is faithful, can be expressed as the map sending $\alpha \in \End(J_{\bar k})$ to the matrix of the pullback $\alpha^*$ in the basis $\{\omega_i\}$,
\[
\End(J_{\bar k}) \hookrightarrow M_{g \times g}(\bar k), \quad \alpha \mapsto [\alpha^\ast \omega_i].
\]
The Galois action of an automorphism $\tau \in \Gal(\bar k/k)$ is thus given by
\[ \tau(\alpha) \mapsto [\tau(\alpha^\ast) \omega_i] = [\tau(\alpha^\ast) \tau(\omega_i)] = \tau[\alpha^\ast \omega_i], \]
where we used that each $\omega_i$ is $k$-rational. It follows that $\tau$ fixes $\alpha$ if and only if it fixes $M_\alpha$.
\end{proof}

As a consequence, from the matrices $M_\alpha$, we can compute the field of definition of the endomorphisms of $J$. 

\begin{remark}
{\label{remark: general case n1}}
For a general abelian variety (that we still denote by $J$), we may replace the computations described in this section with the following assumption.
\begin{enumerate}
    \item{\label{assumption 1: we are given M and A}} We are given the period matrix $\Pi$ to sufficient numerical precision and matrices
    \begin{equation*}
        {\label{eq: matrices of the tangent representation}}
        M_1, \, \dots,\, M_r \in M_{g \times g}(\bar{\Q})
    \end{equation*}
    giving a basis of the $\Z$-algebra $\End(J_{\bar{k}})$ in the tangent representation.
        Corresponding integral matrices $A_1, \, \dots, \, A_r \in M_{2g \times 2g}(\Z)$ satisfying relation \Cref{eq: Period matrix endomorphism relation} can then be easily calculated.
\end{enumerate}
\end{remark}

Under assumption (\ref{assumption 1: we are given M and A}), without loss of generality we can also suppose the following:
\begin{enumerate}
\setcounter{enumi}{1}
    \item{\label{assumption 2: all endomorphisms are defined over k}} All endomorphisms of $J$ are defined over $k$, i.e.~$k \supseteq k(\End(J))$.
\end{enumerate}
Indeed, we can compute $k' = k(\End(J))$ from the tangent representation and replace $J$ by its base-change $J' = J \times_k k'$.
        Recall that $k(\varepsilon_J) \supseteq k(\End(J))$ \cite[Proposition 2.10]{SilvZarhin} so that our final computation of $k(\varepsilon_J)$ is unaffected. 
        From this point on, we work under the assumption $k = k(\End (J))$, and consequently denote by $\End(J)$ the geometric endomorphism ring.

\section{Computing the CM-type}
\label{section: compute the CM type}
We further assume $J$ to have complex multiplication, in the sense of \cite[\S 3]{MilneCM}. By \cite[Proposition 3.6]{MilneCM}, there is an embedding of a CM algebra $E$ (a finite product of CM fields $E_i$) of degree $2g$ into the endomorphism algebra: 
\begin{equation}
    {\label{eq: embedding iota}}
    \iota \colon E \hookrightarrow \End^0(J) \defeq \End(J)\otimes_{\Z}\Q.
\end{equation}
Note that $E$ is not necessarily unique, but in what follows we can work with any CM subalgebra of $\End^0(J)$ of degree $2g$.
We now explain how to compute $\iota E$.

\smallskip

We know from \cite[Corollary 7.4]{MilneCM} that $\iota E$ is a commutative self-centralizing subalgebra of $\End^0(J)$ (there is a small typo in Milne's statement: $F$ should be replaced by $E$).
Having generators $A_1, \, \dots, \, A_r$ for the matrix algebra $\End^0(J)$ allows us to compute its Wedderburn decomposition 
\[ \End^0(J) \simeq M_{n_1 \times n_1} (D_1) \times \dots \times M_{n_t \times n_t}(D_t), \]
as explained in \cite{AlgebraDecomposition}, where each $D_i$ is a simple division algebra. Since we are in the CM case, each $D_j$ is in fact a CM field \cite[Proposition 3.1]{MilneCM}. Concretely, we can assume the generators $A_i$ to be block diagonal, with blocks of dimension $n_1, \dots, n_t$.

To compute (a valid choice of) $\iota E$ it is thus enough to find a maximal self-centralizing CM subalgebra in each simple factor. If a polarization is fixed (and is given explicitly as a bilinear form on $\Q^{2g}$), we may choose the image of $\iota E$ so that it is stable under the Rosati involution corresponding to the given polarization \cite[Proposition 3.6(c)]{MilneCM}. %

\begin{remark}
In practice, we proceed as follows. First, we compute the center of each factor $M_{n_i \times n_i}(D_i)$, which gives the CM field $D_i$. Next, we take the points in $M_{n_i \times n_i}(D_i)$ fixed by the Rosati involution, obtaining a subalgebra $D'_i$. We then compute the centralizer of a \emph{splitting element} $s_i$ of $D'_i$; this gives a maximal self-centralizing subalgebra $E_{i, 0}$ of $D'_i$, which is clearly fixed under the Rosati involution by construction.
A splitting element is a `general enough' element of the matrix algebra, which we compute by taking a random linear combination of the generators and then verifying that its centralizer generates an algebra of the correct dimension, as explained in \cite[\S 2, Observation 1]{MR1758404}. Finally, we take as CM algebra $\iota E$ the product over $i$ of the subalgebras generated by the CM center $D_i$ together with $E_{i, 0}$; both $D_i$ and $E_{i, 0}$ are stable under the Rosati involution, hence so is $\iota E$. Naturally, when each $n_i$ is $1$, proceeding in this way we get $\End^0(J)$ itself.
\end{remark}

We will have no further use for the full endomorphism ring $\End^0(J)$, and will only need its subring $\iota E$. Up to a change of basis (determined by the previous calculation), we can assume that for some $s \leq r$ the matrices $M_1, \ldots, M_s$ generate $\iota E$. We will only work with $M_1, \ldots, M_s$ and the corresponding integral matrices $A_1, \ldots, A_s$.

\subsection{The CM type} The commutative algebra $E$ acts on the tangent space by formally composing $\iota$ with the tangent representation. The resulting abelian representation on $\HH^0(X, \Omega_X^1) \otimes_k \C$ decomposes as the sum of $g$ one-dimensional eigenspaces, indexed by a set of \emph{distinct} characters that we identify with homomorphisms
\begin{equation}
\label{eq: CM type}
    \Phi = \{ \chi_i\colon E \to \C \}.
\end{equation}
The set $\Phi$ has the property that no two homomorphisms are complex-conjugate and therefore $\Phi \sqcup \bar{\Phi} = \Hom(E, \C)$. The set $\Phi = \{ \chi_1, \, \dots ,\,  \chi_g\}$ is called the CM-type of the abelian variety $J$. We set $\chi_{i+g} = \overline{\chi_i}$ for all $i=1,\ldots, g$, so that $\bar{\Phi} = \{\chi_{g+1}, \, \dots, \,\chi_{2g}\}$.
We can compute the CM-type from the tangent representation. %

\begin{remark}
In practice, we proceed as follows. 
Upon restricting to the $i$-th diagonal block, the minimal polynomials of the generators $A_1, \, \dots, \, A_{s}$ determine the extension $E_i/\Q$.
Indeed, one can construct the extension inductively, by considering at the $(j+1)$-th step the minimal polynomial of $A_{j+1}$ over the extension determined by $A_1, \, \dots, \, A_j$.
Alternatively, one can compute the minimal polynomial of a splitting element $s_i$, as explained in \cite{Eberly2}.%

    We can replace the target of the homomorphisms in \Cref{eq: CM type} with the Galois closure $F$ of the compositum of all extensions $E_i/\Q$ (or more generally, any Galois extension of $\Q$ containing each $E_i$). 
    In particular, the eigenvalues of any $M_j$ are contained in $F$.
    The matrices $M_1, \, \dots, \, M_{s}$ commute and are diagonalizable (over $F$). It is a straightforward exercise in linear algebra to determine a simultaneous eigenbasis $\{\omega_j\}_j$, where each $\omega_j$ is an element of $H^0(X, \Omega_X^1) \otimes_k kF$.
    The eigenvalues determine the homomorphisms in \Cref{eq: CM type}.
\end{remark}

A similar argument applies to the {rational representation} $V=\Lambda\otimes_{\Z}\Q$ sending $\alpha \mapsto A_\alpha$. 
The base change $V_{\C}$ always splits as the direct sum of the tangent representation and its complex conjugate.
In the CM case, the base change $V_{F}:= V \otimes_\Q F$ decomposes as the direct sum of one-dimensional representations indexed by $\Hom(E, \C) = \Phi\sqcup \bar{\Phi}$.
We compute an eigenbasis $v_1, \, \dots, \, v_{2g}$ of $V_F$, indexed so that $A_\alpha v_j = \chi_j(\alpha) \cdot v_j$ for all $\alpha \in \iota(E)$.%

\begin{remark}
{\label{remark: general case n2}}
For a general abelian variety $J$, we can replace the computations presented in this section with the following assumptions.
\begin{enumerate}[start = 3]
    \item The abelian variety $J/k$ has complex multiplication and the CM-pair $(E, \Phi)$ is given.
    For a given $s \leq r$, the matrices $M_1, \, \dots, M_s$ generate $\iota E \subseteq \End(A)$.

    \item{\label{assumption 4: eigenbasis omega}} We are given a basis $\omega_1, \, \dots, \, \omega_g$ for $\HH^0(J, \Omega^1){\otimes_k kF}$ such that $\alpha(\omega_i)=\chi_i(\alpha)\cdot \omega_i$ for all $\alpha \in E$. 

    \item We are given an eigenbasis $v_1, \, \dots, \, v_{2g}$ for the rational representation $V_{F}$.
\end{enumerate}
\end{remark}

\section{Computing the Mumford-Tate group}
Let $E$ be a CM field and $\Phi$ be a CM type on $E$. Fix a CM field $F$ that contains $E$ and is Galois over $\Q$ (for example, we can take $F$ to be the Galois closure of $E/\Q$) and let $G$ be the Galois group of $F$ over $\Q$.
The set %
\[\tilde{\Phi} = \{ \chi \in \Hom(F, \C) \mid \chi_{\mid E} \in \Phi \}\]
is a CM type on $F$.
The fixed field of the reflex subgroup $H^\ast = \{g \in G \mid  \tilde\Phi g = \tilde\Phi \} \leq G$ is, by definition, the reflex field $E^\ast$. This is a CM field which comes equipped with the reflex CM-type $\Phi^\ast$, given by the restriction to $E^\ast$ of $\tilde\Phi^{-1} := \{ \chi^{-1} : \chi \in \tilde{\Phi} \}$, where $\chi^{-1}$ is the inverse of $\chi$ seen as an element of the group $G$.
The reflex norm is the map
\begin{equation}
    {\label{eq: reflex norm fields}}
    \textstyle
    N_{\Phi^\ast}\colon E^\ast \to E, \quad x \mapsto \prod_{\chi \in \Phi^\ast} \chi(x). 
\end{equation}

\begin{definition}
    Given a field extension $E/\Q$, consider the algebraic torus $T_E = \Res_{\Q}^{E}\mathbb{G}_{m, E}$, where $\Res_{\Q}^{E}$ denotes the Weil restriction of scalars. The association $E \mapsto T_E$ is functorial.
\end{definition}

We now go back to considering our CM Jacobian $J$. Throughout the rest of the paper, we will work with a field $F$ chosen as follows:
\begin{definition}\label{def: field F}
We let $F$ be a finite Galois extension of $\Q$ that contains each of the components $E_1, \,\ldots,\, E_t$ of the CM algebra $E = E_1 \times \cdots \times E_t$.    
\end{definition}

An application of \cite[Lemma 4.1]{LombardoMT} allows us to compute the Mumford-Tate group $\MT(J)$ as the image of the map
\begin{equation}
    {\label{eq: MT computation tori step}}
    \begin{tikzcd}
        T_{F} \rar["N_{F/E_i^\ast}"] & \prod_i T_{E_i^\ast} \rar["N_{\Phi_i^\ast}"] & \prod_i T_{E_i} \subseteq \GL_V,
    \end{tikzcd}
\end{equation}
where
$N_{F/E_i^\ast}$ is the norm map and 
the torus $\prod T_{E_i}$ is naturally embedded in $\GL_V$ via the rational representation. %

A homomorphism $F_1^\times \to F_2^\times$ between the multiplicative groups of two number fields that is given by a (finite) product of homomorphisms $F_1 \to \mathbb{C}$ gives rise in a natural way to an algebraic morphism $T_{F_1} \to T_{F_2}$. In turn, by exploiting the anti-equivalence between the category of algebraic tori over $\Q$ and the category of free abelian groups with a continuous $\Gal(\bar{\Q}/\Q)$-action, we can interpret such a morphism as a linear map between free $\Z$-modules, that is, as an integral matrix.
Identifying a homomorphism $E \to \mathbb{C}$ to a tuple $(\chi_i)_{i=1,\ldots,t}$ of homomorphisms $\chi_i : E_i \to \mathbb{C}$, all but one trivial (note that a map $E \to \mathbb{C}$ factors via one of the quotients $E_i$ of $E$), we rewrite \Cref{eq: MT computation tori step} as %
\begin{equation*}
    \begin{tikzcd}[column sep = 3.2cm]
        \Z[\Hom(E, \C)]
        \rar["(\chi_i)_i \mapsto \left(\sum_{\sigma \in \Phi^\ast_i} \sigma\chi_i\right)_i"] &
        \Z[\Hom(\prod_i E_i^\ast, \C)]
        \rar["\chi \mapsto \sum_i\sum_{g \in \Gal(F/E_i^*)} g\chi"] &
        \Z[\Hom(F, \C)].
    \end{tikzcd}
\end{equation*}
The kernel of the composition above, which is a free abelian group with a Galois action, corresponds to the torus $\MT(J) \subseteq \GL_V$.

Upon extension of scalars to $F$, the rational representation is diagonal in the basis $\{v_i\}_i$, and we can identify $(T_E)_F  = T_E \times_{\Spec \Q} \Spec F$ to the diagonal torus of $\GL_{V, F}$ (recall that $V$ is a free $E$-module, see \cite[Proposition 3.6(c)]{MilneCM}). Fix coordinates $\{x_i\}_i$ for the diagonal torus. Each element in the kernel %
can then be interpreted as a Laurent monomial $f :=x_1^{e_1} \cdots x_{2g}^{e_{2g}}$ in the variables $x_i$ with the property that $f$ is identically $1$ on the Mumford-Tate group. By fixing a basis of this kernel, we then obtain a finite number of Laurent monomials $f_1, \, \dots, f_r$ defining the coordinate ring of the Mumford-Tate group (as in \cite[Section 4.2]{gallese_part1}):
\begin{equation}\label{eq: eqs for MT}
    \mathcal{O}_{\MT(J)_F} = F[x_1^{\pm 1},\dots, x_{2g}^{\pm 1}]/(f_1-1, \dots, f_r-1). 
\end{equation}

\begin{remark}\label{rmk: MT equations are homogeneous of degree 0}
    Note that each equation $f = \prod_{i=1}^{2g} x_i^{e_i}$ has total degree $\sum_{i=1}^{2g}e_i=0$.
    Indeed, since the Mumford-Tate group contains the homotheties, any Laurent monomial that is identically equal to $1$ on $\MT(J)$ must have total degree zero.
    The \emph{weight} of $f$ is the integer $n = \tfrac{1}{2} \cdot \sum_{i=1}^{2g}|e_i|$.
\end{remark}

\begin{remark}
{\label{remark: general case n3}}
For a general CM abelian variety, we can replace the computations presented in this section with the following assumption.
\begin{enumerate}[start = 6]
    \item{\label{assumption 6: family of MT equations}} We are given a finite generating set of equations $\mathcal{F} = \{f_1, \, \dots, \, f_r\}$ of $\MT(J)$ in the sense of \cite[Definition 4.2.4]{gallese_part1}. Each $f$ is a Laurent monomial $x_1^{e_1}\cdots x_{2g}^{e_{2g}}$.
\end{enumerate}
\end{remark}

\section{Computing the algebra of Hodge classes}
Recall that a \textit{(pure) Hodge structure} of weight~$n$ on a rational vector space $V$ 
is a decomposition of its complexification
\begin{equation}\label{eq: Hodge structure}
    V_{\mathbb{C}} := V \otimes_{\mathbb{Q}} \mathbb{C}
    = \bigoplus_{p+q=n} V^{p,q}
\end{equation}
such that $\overline{V^{p,q}} = V^{q,p}$. 
When $n = 0$, an element of $V^{0,0} \cap V$ is called a \textit{Hodge class}. 
For a complex abelian variety~$A$, the cohomology group $V := \HH^1(A, \mathbb{Q})$ 
carries a natural Hodge structure of weight~$1$, 
with $V^{1,0}$ given by the space of holomorphic $1$-forms on~$A$ 
and $V^{0,1}$ by its complex conjugate. 
A \textit{morphism of Hodge structures} is a morphism of $\mathbb{Q}$-vector spaces whose complexification respects the decomposition~\eqref{eq: Hodge structure}. %
The Hodge structure~$\mathbb{Q}(1)$ is, by definition, the unique one-dimensional Hodge structure of weight~$-2$. 
For each Hodge structure~$V$ and each $n \in \mathbb{Z}$, we denote by $V(n)$ the Hodge structure $V \otimes \mathbb{Q}(1)^{\otimes n}$, 
where $\mathbb{Q}(1)^{\otimes n} := \operatorname{Hom}(\mathbb{Q}(1), \mathbb{Q})^{\otimes |n|}$ for $n < 0$. 
Finally, attached to a Hodge structure~$V$, there is a canonical algebraic subgroup of~$\GL_V$, 
called the \textit{Mumford–Tate group}~$\operatorname{MT}(V)$ of~$V$. 
See~\cite{moonenIntro} for further details.

\smallskip

From this point on, we let $A/k$ be a CM abelian variety that satisfies assumptions (1) to (6) of \Cref{remark: general case n1,remark: general case n2,remark: general case n3}. Let $V:=\HH_1^B(A(\C), \Q)$ be the associated Hodge structure of weight $-1$. The Mumford-Tate group of $A$ is by definition the Mumford-Tate group of $V$. A polarization $A \to A^\vee$ gives an isomorphism $V \simeq V^\vee(1)$ of Hodge structures. 
Note that for any $n \geq 0$ the tensor product $V^{\otimes n} \otimes (V^\vee)^{\otimes n}$ is a pure Hodge structure of weight $0$ and consider the subalgebra of Hodge classes 
\begin{equation}
    \label{eq: algebra of Hodge classes}
    \textstyle W \subseteq \bigoplus_{n \geq 0} V^{\otimes n} \otimes V^{\vee,\otimes n} \simeq \bigoplus_{n \geq 0} V^{\otimes 2n}(-n).
\end{equation}
The algebra structure is induced by the isomorphism $\bigoplus_n V^{\otimes 2n}(-n) \simeq \bigoplus_n V^{\vee,\otimes 2n}(n)$: the right-hand side is a subring of the cohomology algebra, where the cup product is naturally defined.

\begin{definition}
{\label{definition: index tuple for MT equation}}
    Given a Laurent monomial $f = x_1^{e_1}\cdots x_{2g}^{e_g}$ of total degree $\sum_i e_i = 0$, let $I_f$ be the index tuple obtained by concatenating
    \begin{equation*}
        \underbrace{(i, i, \dots, i)}_{e_i \text{ times}}
        \text{   if } e_i>0, \quad\text{ or }\quad
        \underbrace{(c(i), c(i), \dots, c(i))}_{|e_i| \text{ times}} \text{   if } e_i<0,
    \end{equation*}
    for $i=1, \dots, 2g$, where $c(i) = i+g \bmod 2g$.
    This applies in particular to equations $f$ for $\MT(A)$, see \Cref{rmk: MT equations are homogeneous of degree 0}.
\end{definition}

Recall the field $F$ introduced in \Cref{def: field F}.

\begin{definition}
{\label{definition: eigenspaces for the endomorphism action}}
    Given an index $2n$-tuple $I=(i_1, \ldots, i_{2n})$, where each $i_j$ is in $\{1, \dots, \, 2g\}$, we define $V_F(I) \subseteq V^{\otimes 2n}(-n) \otimes F \cong V_F^{\otimes 2n} \otimes_F F \,(2\pi i)^{-n}$ as the $1$-dimensional $F$-subspace spanned by
    \[
    \textstyle v_I \defeq \bigotimes_{i \in I} v_i \; \otimes (2\pi i )^{-n} .
    \]
\end{definition}

In the following lemma, we determine the eigenvalues of the Mumford–Tate group acting on $v_I$.

\begin{lemma}\label{lemma: MT action on tensor powers}
    Let $x$ be a point in $\MT(A)_F(\bar{F}) \subseteq (T_E)_F(\overline{F})$. Recall that we have identified $(T_E)_F$ with the diagonal torus of $\GL_{V, F}$, so we can write $x=\operatorname{diag}(x_1,\ldots,x_{2g})$ for certain $x_1, \ldots, x_{2g} \in \bar{F}^\times$.
    Let $N,\, D$ be two $n$-tuples of indices in $\{1,\, \ldots, \, 2g\}$ and let $I$ be the $2n$-tuple obtained as the concatenation of $N$ and $D$. The element $x$ acts on $V_{F}(I)$ as multiplication by $\prod_{i \in N} x_i \cdot (\prod_{i \in D} x_{c(i)})^{-1}$, where $c(i)$ is as in \Cref{definition: index tuple for MT equation}.
\end{lemma}
\begin{proof}
By definition, $x \in \MT(A){_F(\bar{F})}$ acts on $V_F^{\otimes 2n}(-n) { \otimes_F \bar{F}}$ as $x \otimes \cdots \otimes x \otimes \operatorname{mult}(x)^{-n}$, where $\operatorname{mult}(x)$ is the multiplier of the symplectic matrix $x$ \cite[\S 5.2]{moonenIntro}. Since $v_i,\, v_{c(i)}$ form an $\MT(A)_{F}$-stable symplectic plane for every $i$ {(see also \Cref{lemma: orthogonality differential forms} below)}, the multiplier of $x$ is $x_ix_{c(i)}$ for any $i \in \{1,\ldots,2g\}$. Since $|D|=n$, we have $\operatorname{mult}(x)^{-n} = \prod_{i \in D} \operatorname{mult}(x)^{-1}=\prod_{i \in D} (x_i x_{c(i)})^{-1}$.
Note furthermore that, by our choice of coordinates, $v_i$ is an eigenvector of $x$ with eigenvalue $x_i$.

Combining the above observations, we obtain that the action of $x$ on the generator $v_I = \bigotimes_{i \in I} v_i \otimes 1 = \bigotimes_{i \in N} v_i \otimes \bigotimes_{i \in D}v_i \otimes 1$ sends it to
    \[
    \begin{aligned}
    \bigotimes_{i \in I} xv_i \otimes \operatorname{mult}(x)^{-n}
    & = \bigotimes_{i \in N} x_i v_i \otimes \bigotimes_{i \in D} x_i v_i \otimes  \prod_{i \in D} (x_i x_{c(i)})^{-1} \\
    & = \prod_{i \in N} x_i \cdot \prod_{i \in D} \frac{x_i}{x_ix_{c(i)}} \bigotimes_{i \in N}v_i \otimes \bigotimes_{i \in D} v_i \otimes 1 \\
    & = \prod_{i \in N} x_i \cdot \prod_{i \in D} x_{c(i)}^{-1} \cdot v_I,
    \end{aligned}
    \]
    as desired.
\end{proof}

\begin{proposition}
    {\label{prop: W is generate by VIf}}
    Let $A$ be an abelian variety with complex multiplication and $W\subseteq \bigoplus_{n} V^{\otimes 2n}(-n)$ be the algebra of Hodge classes.%
    \begin{enumerate}
        \item The inclusion $V_F(I) \subseteq W_F$ holds if and only if there is an equation $f$ for $\MT(A)$ such that $I = I_f$. Conversely, given a Laurent monomial $f$, the inclusion $V_F(I_f) \subseteq W_F$ holds if and only if $f$ is an equation for $\MT(A)$.
        \item Let $\mathcal{F} =\{f_1,\ldots,f_r\}$ be a finite set of equations for $\MT(A)_F$, in the sense of \Cref{eq: eqs for MT}. The subspace $W'_F = \sum_{f \in \mathcal{F}} V_F(I_f)$ generates $W_F$.
    \end{enumerate}
\end{proposition}
\begin{proof}
    Write the $2n$-tuple $I$ as the concatenation of two $n$-tuples $N,\, D$. 
    One associates with $I$ the monomial 
    \[ \textstyle f = \prod_{i \in N} x_i \cdot \left(\prod_{i \in D} x_{c(i)}\right)^{-1}. \]
    The condition $V_F(I) \subseteq W_F$ is equivalent to the fact that $V_F(I)$ is left pointwise stable by the action of $\MT(A)_F$ \cite[\S 4.4]{moonenIntro} or, equivalently, by the action of all geometric points $x \in \MT(A)_F(\bar{F})$. 
    Since by \Cref{lemma: MT action on tensor powers} the element $x$ acts on the $1$-dimensional vector space $V_F(I) \otimes_F \bar{F}$ as multiplication by $f(x)$, this is in turn equivalent to $f(x)=1$ for all $x \in \MT(A)_F(\bar{F})$. 
    This is also equivalent to the fact that $f-1$ lies in the ideal defining $\MT(A)_F$ as an algebraic subgroup of the diagonal torus.
    This proves (1).
    Part (2) follows easily.%
\end{proof}

\section{Computing the Galois action on de Rham homology}
Let $V^\dR_k = \HH^1_\dR(A/k)^\vee$.
We identify the 1-forms $\omega_i$ of assumption (\ref{assumption 4: eigenbasis omega}) with their image through the embedding $\HH^0(A, \Omega_A^1) \otimes_k kF \hookrightarrow \HH^1_{\dR}(A/kF)$.
Complete $\{\omega_i\}_{i=1,\ldots,g}$ to a basis $\{\omega_i\}_{i=1,\ldots,2g}$ of $\HH^1_{\dR}(A/kF)$ consisting of eigenvectors for the $E$-action and let $\{\omega_i^\vee\}_{i=1,\ldots,2g}$ be the dual basis of $V_k^{\dR}$. It is immediate to check that $\omega_i^\vee$ is an eigenform with the same character as $\omega_i$.

The Galois action on $V^\dR_{k}\otimes_k kF$ is the diagonal one: the tensor representation of $V_{\dR, k}$ with the trivial action and $kF$ with its natural Galois action. By assumption (\ref{assumption 2: all endomorphisms are defined over k}) we have
\begin{equation}
    \label{eq: Galois action on line}
    \alpha \cdot \tau(\omega_i^\vee) = \tau(\alpha \cdot \omega_i^\vee) = \tau\chi_i(\alpha) \cdot \tau(\omega_i)^\vee \quad \forall \alpha \in E, \, \forall i = 1, \dots, 2g.
\end{equation}
Define a $\Gal(kF/k)$-action on the index set $\{1, \dots, 2g\}$ by the formula $\tau\chi_i = \chi_{\tau(i)}$.

\begin{lemma}
\label{lemma: rescaling of the eigenbasis}
    Up to rescaling the $\omega_i^\vee$, we can assume that
    \[ \tau(\omega_i^\vee) = \omega_{\tau(i)}^\vee \qquad \forall \tau \in \Gal(kF/k). \]
\end{lemma}
\begin{proof}
    Let $L=kF$ and $G = \Gal(L/k)$.
    Every $\tau \in G$ acts on $V^\dR_k$ as a generalized permutation matrix in the basis $\{\omega_i^\vee\}$ by \Cref{eq: Galois action on line}. We can decompose the matrix representing the $\tau$-action as the product of a permutation matrix and a diagonal matrix with entries $D(\tau)$ (where $D$ is the left factor).
    Consider the map
    \[ G \to L^{\times 2g}, \quad \tau \mapsto D(\tau). \]
    Associativity of the Galois action translates into the cocycle condition
    \[ D(\sigma\tau) = D(\sigma)\cdot \sigma D(\tau) \qquad \forall \sigma, \, \tau \in G, \]
    where $\sigma$ acts on a vector $D$ by permuting its components and acting on the coefficients. We consider $L^{\times 2g}$ as a Galois module with this action, so that $D \in \HH^1(G, L^{\times 2g})$.
    Decompose the index set $\{1, \dots, 2g\}$ into 
    $G$-orbits; pick a representative $i$ in each orbit and let $H_i \leq G$ be its stabilizer. There is a $G$-module decomposition
    \[ \textstyle L^{\times 2g} = \oplus_i \Ind_{H_i}^G (L^\times). \]
    From Shapiro's Lemma and Hilbert 90, we deduce that
    \[ 
        \HH^1(G, \oplus_i \Ind_{H_i}^G L^\times)
        = \oplus_i \HH^1(G, \Ind_{H_i}^G L^\times) 
        = \oplus_i \HH^1(H_i, L^\times)
        = 0.
    \]
    It follows that $D$ is trivial in cohomology, hence there exists $x \in L^{\times2g}$ such that $D(\tau)$ is given by $\tau(x)/x$. The diagonal change of basis with entries $x^{-1}$ gives the desired result.
\end{proof}

Note that computing a basis as in \Cref{lemma: rescaling of the eigenbasis} is easy: since each $E$-eigenspace is $1$-dimensional, it suffices to rescale $\omega_i^\vee$ so that its first coefficient with respect to a fixed $k$-rational basis is $1$ (indeed, $\tau(\omega_i^\vee)$ is proportional to $\omega_{\tau(i)}^\vee$, and the fact that the first coefficient in a basis expansion is $1$ for both implies that they are equal).

From this point on, we assume that our basis $\{\omega_i^\vee\}$ is normalized as in \Cref{lemma: rescaling of the eigenbasis}.

\begin{definition}
    {\label{definition: differential forms line decomposition}}
    Let $I = (i_1, \dots, i_{2n})$ be a $2n$-tuple as in \Cref{definition: eigenspaces for the endomorphism action}.
    The Galois action on the index set $\{1, \dots, 2g\}$ induces an action on $2n$-tuples by setting $\tau I \defeq (\tau i_1, \dots, \tau i_{2n})$.
    Write $V_F^{\dR} := V_k^{\dR} \otimes_k F$ and denote by $V^\dR_{F}(I) \subseteq V_{F}^{\dR, \otimes 2n}(-n)$ the $F$-line generated by 
    \[ \textstyle \omega_I^\vee \defeq \bigotimes_{i \in I} \omega_i^\vee. \]
    Let $\omega_f^\vee \defeq \omega_{I_f}^\vee$.
\end{definition}

It follows from \Cref{lemma: rescaling of the eigenbasis} that $\tau(\omega_I^\vee) = \omega_{\tau I}^\vee$ and $\tau V_{F}^\dR(I) = V_{F}^\dR(\tau I)$. %

\begin{definition}
\label{definition: big eigenspace}
    Let $I$ be a $2n$-tuple and denote by $[I]$ the orbit of $I$ under the Galois action. Set
    \[
    V_F[I] = \sum_{J \in [I]} V_F(J) = \sum_{\tau \in \Gal(F/k)} V_F(\tau I).
    \]
Notice that $V_F[I]$ is a finite-dimensional vector space and that $\tau V_F[I] = V_F[I]$. It follows that $V_F[I]$ is the $F$-span of the rational vector space $V[I] = V_F[I] \cap (V \otimes_\Q 1)$.
We use the notation $V_{k}^\dR[I]$ for the analogous subspaces of $V_{k}^\dR$.
\end{definition}

\section{Computing the Galois action on Hodge cycles}
Absolute Hodge cycles were introduced by Deligne in \cite[Section 2]{Deligne} as classes in the total cohomology space $\HH_{\mathbb{A}}^\bullet$, defined formally as the product of the de Rham cohomology algebra $\HH_{\dR}^\bullet$ and the étale $\ell$-adic cohomology alegbra for all primes $\ell$.
By definition, an absolute Hodge cycle $\gamma$ on a projective variety $X$ has a de Rham component $\gamma_{\dR} \in \HH^{2n}_{\dR}(X/k)(n)$ and a component in each $\ell$-adic étale cohomology group, $\gamma_\ell \in \HH^{2n}_{\text{ét}}(X_{\bar{k}}, \Q_\ell)(n)$. Components are required to be compatible through the various comparison isomorphisms. The space of absolute Hodge cycles is a $\Q$-vector space $C_{\operatorname{AH}} \subseteq \HH^\bullet_{\mathbb{A}}$.

We now go back to working with our CM abelian variety $A$.
For consistency with the previous sections, we work with the dual space $C_{\AH}^\vee$ of absolute Hodge cycles.
The $E^{2n}$-action on the different cohomology groups is compatible with the comparison isomorphisms, so $C_{\AH}^\vee \otimes_\Q F$ decomposes into generalized eigenspaces $C_F[I]$ as in \Cref{definition: big eigenspace}. Notice that $C_F[I]$ can be trivial, even when $V_F[I]$ is not, as the latter might not contain any Hodge class.

\begin{proposition}
    \label{prop: W is the etale realisation of CAH}
    Let $W$ be the algebra of Hodge classes. Fix an embedding $k \hookrightarrow \C$.
    The inclusion $W \to C_{\AH}^\vee$ induced by the embedding is an isomorphism.
    In particular, $C_F[I]$ is nontrivial if and only if $I=I_f$ for an equation $f \in \mathcal{F}$.
\end{proposition}
\begin{proof}
    For abelian varieties, Hodge classes are absolute Hodge \cite[Main Theorem 2.11]{Deligne}. %
    The second part of the statement follows from \Cref{prop: W is generate by VIf}.
\end{proof}

Fix an equation $f \in \mathcal{F}$.
The space $C_F[I_f]$ is Galois stable, hence the $F$-span of the rational subspace $C[I_f] \subseteq C_F[I_f]$ of absolute Hodge cycles.

\begin{corollary}
\label{corollary: eigenspace decomposition of CAH}
    The $\Q$-algebra $C_{\AH}^\vee$ is generated by the finite-dimensional $\Q$-vector space $\sum_{f \in \mathcal{F}} C[I_f]$.
\end{corollary}
\begin{proof}
    Follows from the isomorphism in \Cref{prop: W is the etale realisation of CAH} and \Cref{prop: W is generate by VIf}(2).
\end{proof}

\begin{remark}
    Note that if $f$ is an equation for $\MT(A)$ in the sense of \Cref{eq: eqs for MT}, then there exists an equation $f'$ such that $\tau I_f = I_{f'}$. Indeed, since $f$ is an equation for $\MT(A)$, the space $C_F(I_f)$ consists of absolute Hodge classes;
    absolute Hodge classes are stable under the Galois action by definition, so $\tau(C_F(I_f)) = C_F(\tau I_f)$ is also nontrivial
    , which -- using \Cref{prop: W is the etale realisation of CAH} -- implies that $\tau I_f = I_{f'}$ for some equation $f'$ for $\MT(A)$. We set $\tau f = f'$.
\end{remark}

The principal advantage of working with absolute Hodge cycles is that the action of the Galois group 
$\Gal(\bar k/k)$ can be computed on the de Rham realization, where it can be described in terms of the period integrals $P(\sigma, \omega)$ (see \Cref{definition: periods}).
\begin{proposition}
    \label{proposition: CI is Galois stable iff period condition}
    Fix an equation $f \in \mathcal{F}$ of weight $n$. Let $\lambda$ be a topological cycle that generates $V = \HH_1^B(A(\C), \Q)$ as a free $E$-module \cite[Proposition 3.6(c)]{MilneCM}. Let $\sigma \in \HH^B_{2n}(A^{2n}(\C), \Q)$ be the class corresponding to $\lambda^{\otimes 2n} $ via the {K\"unneth} product formula. The period $P(\sigma, \omega_J)$ is algebraic for all $J \in [I_f]$.
    The space $C[I_f]$ is fixed pointwise by the Galois action if and only if
    \begin{equation}\label{eq: period conditions}
    \tau P(\sigma, \omega_f) = P(\sigma, \omega_{\tau f}) \qquad \forall \tau \in \Gal(\bar k/k).        
    \end{equation}
    Here, we consider $\omega_f \in \HH^1_\dR(A/kF)^{\otimes 2n}$ as a differential form on $A^{2n}$ via K\"unneth's formula. %
\end{proposition}
\begin{proof}
Let $I=I_f$ for ease of notation.
Through the singular-to-de Rham isomorphism
\[
    b \colon V^{\otimes 2n}\hookrightarrow V^{\otimes 2n} \otimes_\Q \C(-n) \simeq V_{F}^{\dR, \otimes 2n} \otimes_F \C(-n),
\]
the image of the cycle $\sigma$ has $[I]$-component
\[ \operatorname{proj}_{[I]}b(\sigma) \defeq
    {\sum_{J \in [I]} \omega_{J}^\vee \otimes P(\sigma, \omega_{J}) } \in V^\dR_F[I] \otimes_F \C. \]
Since $V^\dR_{\C}[I] \simeq C[I] \otimes_\Q \C$ is generated by (the de Rham components of) absolute Hodge cycles and $\operatorname{proj}_{[I]}b(\sigma)$ is rational, $\operatorname{proj}_{[I]}b(\sigma) \in C[I]$ must be (the de Rham component of) an absolute Hodge cycle.
Notice that $P(\sigma, \omega_J) = P(b^{-1} \operatorname{proj}_{[I]}b(\sigma), \omega_J)$. Since $b^{-1} \operatorname{proj}_{[I]}b(\sigma)$ maps to absolute Hodge cycle, it follows from \cite[Proposition 7.1]{Deligne} that $P(\sigma, \omega_J)$ is an algebraic number for every $J \in [I_f]$.
The Galois action on $V_F^{\dR} \otimes_F {\bar k}$ is diagonal. By \Cref{lemma: rescaling of the eigenbasis} we get
\[ \tau(\operatorname{proj}_{[I]}(b\sigma))
    = \sum_{J \in [I]} \omega_{\tau J}^\vee \otimes \tau P(\sigma, \omega_J) \]
and therefore $\tau (\operatorname{proj}_{[I]}(b\sigma)) = \operatorname{proj}_{[I]}(b\sigma)$ if and only if
\begin{equation}
    {\label{eq: shifting condition for all J}}
    \tau P(\sigma, \omega_J) = P(\sigma, \omega_{\tau J}) \qquad \forall J \in [I].
\end{equation}

By letting $\sigma$ vary in its $E^{2n}$-orbit, we obtain a basis of $V^{\otimes}$ and therefore
the elements $\proj_{[I]}(b\sigma)$ form a $\Q$-generating set for the vector space $C[I]$.  %
We claim that if \Cref{eq: shifting condition for all J} holds, it also holds replacing $\sigma$ with $\sigma' = \alpha\cdot \sigma$ for any $\alpha=(\alpha_1, \dots, \alpha_{2n}) \in E^{2n}$.
The coefficient of $\omega_{\tau J}^\vee$ in $\operatorname{proj}_{[I]}(\sigma')$ is %
\[
\begin{aligned}
P(\sigma', \omega_{\tau J}) & = \chi_{\tau j_1}(\alpha_1) \cdots \chi_{\tau j_{2n}}(\alpha_{2n}) P(\sigma, \omega_{\tau J}) \\ & = (\tau \chi_{j_1})(\alpha_1) \cdots (\tau\chi_{j_{2n}})(\alpha_{2n}) P(\sigma, \omega_{\tau J}) \\
& =  \tau\left( \chi_{j_1}(\alpha_1) \cdots \chi_{j_{2n}}(\alpha_{2n}) \right) P(\sigma, \omega_{\tau J}).
\end{aligned}
\]
by the change of variables formula and the fact that $\omega_{\tau J}$ is an eigenvector of character $\chi_{\tau j_1} \otimes \cdots \otimes \chi_{\tau j_{2n}}$. The claim is now easily checked. It follows that $C[I]$ is fixed by the Galois action if and only if \Cref{eq: shifting condition for all J} holds for the single cycle $\sigma$.

One implication is now clear: if $C[I_f]$ is fixed pointwise, then \eqref{eq: shifting condition for all J} must in particular hold for $J=I$ and all $\tau \in \Gal(\bar{k}/k)$, which gives \eqref{eq: period conditions} -- note that $\omega_I = \omega_f$ by definition. For the other, suppose that $\tau P(\sigma, \omega_f) = P(\sigma, \omega_{\tau f})$ for all $\tau \in \Gal(\bar k /k)$. For any $J \in [I_f]$ there is a $\gamma \in \Gal(\bar k/k)$ such that $J =\gamma I_f$. Then
\[ \tau P(\sigma, \omega_{J}) = \tau P(\sigma, \omega_{\gamma f}) = \tau \gamma P(\sigma, \omega_{f}) = P(\sigma, \omega_{\tau \gamma f}) = P(\sigma, \omega_{\tau J}), \]
where we applied our assumption with $\gamma$ and $\tau\gamma$ in the two middle equalities.
\end{proof}

\begin{remark}
\label{remark: if k acts trivially on character then old statement}
    If $k \supseteq F$, then the Galois action on the characters $\chi_i$ is trivial. Thus, $\tau f = f$ for all equations $f \in \mathcal{F}$ and all Galois elements $\tau$, and $k(\varepsilon_A)/k(\End A)$ is generated by the algebraic numbers $P(\sigma, \omega_f)$, as $f$ varies in $\mathcal{F}$. %
\end{remark}

\begin{remark}
\label{remark: Galois action without assumption 2}
    Note that the formulas in the proof of \Cref{proposition: CI is Galois stable iff period condition} give in particular an explicit description of the action of $\operatorname{Gal}(\bar{k}/k)$ on $C[I_f]$.
    Without assumption (\ref{assumption 2: all endomorphisms are defined over k}), the same computations are still possible, and just slightly different from the ones presented. Indeed, the Galois group acts on the characters $\chi_i$ not via post-composition, but rather via
    \begin{equation}\label{eq: more general action on characters}
    (\tau \cdot \chi_i)(\alpha) := \tau(\chi_i(\tau^{-1}\alpha)) \quad \forall \alpha \in E,
    \end{equation}
    where $\tau^{-1} \alpha$ denotes the action of $\tau^{-1}$ on $\alpha$ seen as an element of $\operatorname{End}(A_{\bar{k}})$.
    Notice that \eqref{eq: more general action on characters} restricts to the action $(\tau \cdot \chi_i)(\alpha) = \tau(\chi(\alpha))$ for $\tau \in \operatorname{Gal}(\bar{k}/k(\End E))$.
\end{remark}

\begin{remark}
\label{remark: if k is Q then Kconn is generated by periods}
    For an abelian variety defined over $k = \Q$ with (potential) complex multiplication, we always have $\Q(\End A) \supseteq F$. Indeed, on the one hand we have the containment $\Q(\End A) \supseteq E^\ast$ \cite[Propositon 1.21]{MilneCM}, and on the other, the extension $k(\End A)/k = \Q(\End A)/\Q$ must be Galois, which implies that $\Q(\End A)$ contains the Galois closure of $E^\ast$, which we can take to be our $F$.
    It follows that \Cref{remark: if k acts trivially on character then old statement} always applies.
\end{remark}

\section{Computing the Galois action on Tate classes}
Fix an auxiliary prime $\ell$. %
Consider the canonical étale-to-singular comparison isomorphism \cite[Chapter III, Theorem 3.12]{Milne_etalecohomology}
\begin{equation}\label{eq: Vell F and V F}
V_\ell \simeq \HH^1_{\etale}(A, \Q_{\ell})^\vee \simeq [\HH^1_{B}(A, \Q) \otimes_{\Q} \Q_{\ell}]^\vee \simeq V \otimes_\Q \Q_\ell. %
\end{equation}
Tate classes are elements in $\ell$-adic cohomology fixed by an open subgroup of the absolute Galois group. They form the arithmetic counterpart of Hodge classes and are conjecturally spanned by algebraic cycles, as predicted by the Tate conjecture. 
Denote by $W_\ell$ the algebra of Tate classes in $\bigoplus_n V_\ell^{\otimes n} \otimes V_\ell^{\vee, \otimes n} \simeq V_\ell^{\otimes 2n}(-n)$, described in \cite[\S 2]{gallese_part2}.

\begin{proposition}
\label{proposition: Well is the ladic realization of CAH}
    The $\ell$-adic realization map induces an isomorphism $C_{\AH}^\vee \otimes_\Q \Q_\ell \to W_\ell$ of Galois representations.
\end{proposition}
\begin{proof}
    Composing the realization map with the inclusion $W \to C_{\AH}^\vee$, we get by definition the singular-to-étale comparison map, induced on the product by the isomorphism in \Cref{eq: Vell F and V F}. The image of a Hodge class through the isomorphism is known to always be a Tate class \cite[Theorem 2.11]{Deligne}. Hence, the following diagram is commutative
    \begin{equation*}
    \begin{tikzcd}
        W \otimes_\Q \Q_\ell \rar\dar & C_{\AH}^\vee \otimes_\Q \Q_\ell \arrow[dl] \\
        W_\ell. & 
    \end{tikzcd}
    \end{equation*}
    The Mumford-Tate conjecture, which holds for abelian varieties with complex multiplication \cite{MTconjMC}, implies that the vertical map is an isomorphism. The horizontal map is an isomorphism by \Cref{prop: W is the etale realisation of CAH}. The statement follows. 
\end{proof}

Note that we have now established \Cref{corollary: eigenspace decomposition of CAH}, \Cref{proposition: CI is Galois stable iff period condition}, and \Cref{proposition: Well is the ladic realization of CAH}, which concludes the proof of \Cref{theorem: computation of kconn}.

\section{Computing the periods}
{\label{section: computing periods}}
With \Cref{theorem: computation of kconn} in hand, it only remains to compute the periods $P(\sigma, \omega_f)$ appearing in its statement.
Picking any $\lambda \in \HH_1^B(A, \Q)$ and $\sigma = \otimes_{i \in I_f} \lambda$, we can write the period as a product of integrals
\[
    \int_\sigma \omega_f = \prod_{i \in I_f} \int_{\lambda} \omega_i.
\]
When $A=J$ is a Jacobian, the method of \cite{molin:hal-02416012} can then be used to compute the periods $\int_{\lambda} \omega_i$ to arbitrary precision. Unfortunately, the available implementations only compute the \textit{holomorphic} periods (that is, those with $\omega_i \in \HH^0(J, \Omega^1)$), while we need to know the periods corresponding to a full basis of $\HH_\dR^1(X/F)$. 
Since accurate computation of periods is difficult to implement in practice, especially for anti-holomorphic differentials, in this section we present an alternative approach specific to the CM case. This method yields the quasi-periods (periods of anti-holomorphic forms) without requiring additional numerical computations beyond those already necessary to determine the holomorphic periods.

\subsection{Algebraic relations between periods and quasi-periods}
We prove the following result, which is a refined version of a result of Bertrand \cite[§8, p.~36, equation (3)]{MR702188}. For simplicity, we state it in the case of complex multiplication by a field (the isotypic case). In the general case, if $A \sim \prod A_i$ where the $A_i$ are isotypic, one can apply the same result to each isotypic component by choosing a basis of $\HH^0(A, \Omega^1)$ that is given by the union of bases of the various $\HH^0(A_j, \Omega^1)$.
\begin{theorem}
{\label{theorem: algebraic relation between period and quasiperiod}}
    Let $A$ be an abelian variety over a number field $k$, with complex multiplication by a CM field $E$ with maximal real subfield $E_0$, Galois closure $F$, and CM type $\Phi = \{\chi_1, \ldots, \chi_g\}$. %
    Write $E=E_0(\xi)$, where $\xi \in E$ satisfies $\xi^2 \in E_0$.
    Assume that all the endomorphisms of $A$ are defined over $k$ and that the imaginary part of $\chi_j(\xi)$ is negative for each $j=1,\ldots,g$. Let $\omega_{1}, \ldots, \omega_{g}$ be a basis of $\HH^0(A_{kF}, \Omega^1)$ such that each $\omega_i$ is an eigenform with character $\chi_i$ for the action of $E$. There exist
    \begin{enumerate}
        \item elements $\omega_{g+1}, \ldots, \omega_{2g} \in \HH^1_{\dR}(A/kF)$ such that $\omega_{g+j}$ is an eigenform with character $\bar{\chi}_j$ for all $j=1, \ldots, g$ %
        \item a non-zero cycle $\lambda \in \HH_1(A,\C)$
    \end{enumerate}
    such that for all $j=1, \, \ldots, \,g$ the following identity holds:
    \[
    P(\lambda, \omega_j)\cdot P(\lambda, \omega_{c(j)})= \pi i \cdot \chi_j(\xi).
    \]
\end{theorem}

The same result can be formulated in several alternative ways, some of which may be more amenable to computations. We provide one such translation in \Cref{section: computing periods in the hard case}. 
For the proof we will need the following simple lemma:
\begin{lemma}\label{lemma: orthogonality differential forms}
Let $A/k$ be a CM abelian variety with an action 
$\iota \colon E \hookrightarrow \End^0(A)$ of the CM field $E$.
Let $\psi$ be a polarization on $A$ and let 
$\langle \ , \ \rangle_\psi$ denote the induced alternating form on 
$\HH^1_{\dR}(A/k)$. Suppose that the Rosati involution induced by $\psi$ stabilises $\iota E$.
If $\omega_1,\, \omega_2 \in \HH^1_{\dR}(A/k)$ are eigenforms for the $E$-action with
characters $\chi_1, \chi_2 \colon E \to \overline{k}$ and $\chi_2 \neq \overline{\chi_1}$, then
$
  \langle \omega_1 , \omega_2\rangle_\psi = 0.
$
\end{lemma}

\begin{proof}
The Rosati involution $^{\dagger}$ on $\End^0(A)$ corresponding to $\psi$ restricts to (the unique) complex conjugation on $E$, that is, $\iota(e)^{\dagger} = \iota(\overline{e})$ for all $e \in E$. Hence
for all $e\in E$ we have
\[
  \chi_1(e)\,\langle \omega_1,\omega_2\rangle_\psi
   = \langle \iota(e)\omega_1 , \omega_2\rangle_\psi
   = \langle \omega_1, \iota(\overline{e})\omega_2\rangle_\psi
   = \chi_2(\overline{e}) \,\langle \omega_1, \omega_2\rangle_\psi.
\]
If $\langle \omega_1,\omega_2\rangle_\psi \neq 0$, this forces 
$\chi_1(e) = \chi_2(\overline{e})$ for all $e \in E$, hence $\chi_2=\overline{\chi_1}$ as claimed.
\end{proof}

\begin{proof}[Proof of \Cref{theorem: algebraic relation between period and quasiperiod}]
Part of the argument already appears in \cite{MR702188}, but for the reader's convenience, we prefer to give a mostly self-contained proof. The key ingredient is a comparison between different cohomology theories, so to make it more transparent it is useful to use the notation $\HH_{1}^\dR(A/k)$ for $\HH^1_{\dR}(A/k)^\vee$.

It is well-known that $\HH_1(A(\C),\Q)$ is a free $E$-module of rank 1 \cite[Proposition 3.6(c)]{MilneCM}. Fix an element $\lambda \in \HH_1(A(\C),\Q)$ that generates it as an $E$-module. %
We identify $E$ with $\HH_1(A(\C), \Q)$ via the map $\alpha \mapsto \alpha \cdot \lambda$.

Let $\alpha_1, \ldots, \alpha_g$ be a $\Q$-basis of $E_0$ and $\beta_1, \ldots, \beta_g$ be the dual basis with respect to the trace pairing $(\alpha, \beta) \mapsto \operatorname{tr}_{E_0/\Q}(\alpha\beta)$. Denote by $c$ the complex conjugation of $E$. It is easy to see that $\alpha_1, \ldots, \alpha_g, \frac{\beta_1}{c(\xi)}, \ldots, \frac{\beta_g}{c(\xi)}$ is a $\Q$-basis of $E$, hence 
\[
a_1 := \alpha_1 \cdot \lambda, \, \ldots, \, a_g := \alpha_g \cdot \lambda, \quad b_1 := \frac{\beta_1}{c(\xi)} \cdot \lambda, \, \ldots, \, b_g := \frac{\beta_g}{c(\xi)} \cdot \lambda
\]
is a $\Q$-basis of $\HH_1(A(\C), \Q)$. We also write $\lambda_i := \alpha_i \cdot \lambda$ for $i=1,\ldots,g$ and $\lambda_{i+g} := \beta_i \cdot \lambda$ for $i=1,\ldots,g$.
By \cite[Example 2.9]{MilneCM}, the assumptions on $\xi$ imply that the bilinear form $(\gamma \cdot \lambda, \gamma' \cdot \lambda) \mapsto \operatorname{tr}_{E/\Q}(\gamma \xi c(\gamma'))$ is a Riemann form on $\HH_1(A(\C),\Q)$, hence induces a polarization $\psi : A \to A^\vee$. Moreover, the Rosati involution with respect to this polarization stabilizes $E$.

Note now that any two polarizations $\psi_1,\, \psi_2 : A \to A^\vee$ satisfy $\psi_2^{-1} \circ \psi_1 \in \End^0(A)$. Since there is at least one polarization $A \to A^\vee$ defined over $k$, and all endomorphisms of $A$ are defined over $k$ by assumption, it follows that $\psi$ is defined over $k$.
Moreover, the definition of $\{\beta_j\}$ as the dual basis of $\{\alpha_j\}$ with respect to the trace pairing implies that -- in the basis $a_1, \ldots, a_g, b_1, \ldots, b_g$ -- the Riemann form is represented by 
\begin{equation}\label{eq: definition of matrix J}
    2J := \begin{pmatrix}
    0 & 2I \\ -2I & 0
    \end{pmatrix}.
\end{equation}

The polarization $\psi$ corresponds to an ample line bundle $L$ on $A$ defined over $k$, hence to a Chern class $c_1(L) \in \HH^2_{\dR}(A/k)(1)$, which we can see as an element of $\HH_{2, \dR}(A/k)^\vee (1) \cong \wedge^2 \HH_{1, \dR}(A/k)^\vee (1)$, that is, as an alternating bilinear form $\langle \cdot, \cdot \rangle_\psi$ on $\HH_{1, \dR}(A/k)$. 
Recall the function $c$ of \Cref{definition: index tuple for MT equation} and let $\omega_{c(1)},\, \ldots,\, \omega_{c(g)} \in \HH^1_{\dR}(A/kF)$ be eigenforms for the characters $\bar{\chi}_1, \,\ldots, \,\bar{\chi}_g$. These eigenforms are defined over $kF$, because the $E$-action is diagonalizable over $F$, which contains the images of all the characters $\chi_j,\, \bar{\chi}_j$. Let $\{\omega_i^\vee\}_{i=1,\ldots,2g}$ be the corresponding dual basis of $\HH_{1, \dR}(A/k)$.
 By \Cref{lemma: orthogonality differential forms}, the element $\omega_{c(i)}^\vee$ pairs nontrivially only with $\omega_i^\vee$ under $\langle \cdot, \cdot\rangle_\psi$. Since $\langle \cdot, \cdot\rangle_\psi$ and all the forms $\{\omega_j\}_{j=1,\ldots,2g}$ are $kF$-rational, we can rescale $\omega_{c(j)}$ by a suitable scalar in $F^\times$ to ensure 
 \[
 \langle \omega_j^\vee, \omega_{c(j)}^\vee\rangle_\psi = \begin{cases}
2, \text{ for }j=1,\ldots,g \\
-2, \text{ for }j=g+1,\ldots,2g.
\end{cases}
\]
In particular, the matrix representing the polarization with respect to the basis $\{\omega_i^\vee\}_{i=1,\ldots,2g}$ is again $2J$.
We further let
\[
\Omega = \left( \int_{\lambda_j} \omega_{i} \right)_{\substack{i=1,\ldots,g \\ j=1,\ldots,2g}} \quad \text{and} \quad H = \left( \int_{\lambda_j} \omega_{c(i)} \right)_{\substack{i=1,\ldots,g \\ j=1,\ldots,2g}}
\]
be the period and quasi-period matrices in our bases, and set $\Pi = \begin{pmatrix}
    \Omega \\ H
\end{pmatrix}$. 

Now consider the ample line bundle $L$ (corresponding to the polarization $\psi$) as a divisor on $A$ and apply the cycle class maps in singular and de Rham cohomology -- in the case of divisors, this is simply the first Chern class. Compatibility of the cycle class map in various cohomology theories (see \cite[p.~21]{Deligne}) implies that 
\[
\HH^2_B(A(\C), \Q)(1) \otimes_\Q \C \ni \operatorname{cl}_B(L) = \operatorname{cl}_{\dR}(L) \in \HH_{\dR}^2(A/kF)(1) \otimes_{kF} \C,
\]
where the equality is mediated by the comparison isomorphism 
\[
\HH^2_B(A(\C), \Q)(1) \otimes_\Q \C \cong \HH^2_{\dR}(A/kF)(1) \otimes_{kF} \C.
\]
In particular, if we denote by $f : \HH_1^B(A(\C),\Q) \otimes \C \to \HH_{1,\dR}(A/kF) \otimes_{kF} \C$ the singular-to-de Rham isomorphism, by $\langle \cdot, \cdot \rangle_{\psi, B}$ the Riemann form on $\HH_1^B(A(\C), \Q)$, and by $\langle \cdot, \cdot \rangle_{\psi, \dR}$ the bilinear form induced by $\psi$ on $\HH_{1,\dR}(A/kF)$, we have
\begin{equation}\label{eq: de Rham and Betti classes}
    2\pi i \langle x, y\rangle_{\psi, B} =
 \langle f(x), f(y)\rangle_{\psi, \dR}
\end{equation}
for all $x, y \in \HH_1(A(\C), \Q)$. 
We now write this equality in coordinates. A short calculation shows that for $j=1,\ldots,2g$ we have
\[
f(\lambda_j) = \sum_{i=1}^{2g} \Pi_{ij} \omega_i^\vee.
\]
Applying \Cref{eq: de Rham and Betti classes} to $x=\lambda_j, y=\lambda_k$, we obtain 
\[
\begin{aligned}
2\pi i (2J)_{j, k} & = 2\pi i \langle \lambda_j, \lambda_k \rangle_{\psi, B} =
\langle f(\lambda_j), f(\lambda_k) \rangle_{\psi, \dR} \\
& = \sum_{\ell, m} \Pi_{\ell, j} \Pi_{m, k} \langle \omega_\ell^\vee, \omega_m^\vee \rangle_{\psi, \dR} = \sum_{\ell, m} {}^t\Pi_{j,\ell} (2J)_{\ell, m} \Pi_{m, k} \\
& = ({}^t \Pi \cdot 2J \cdot \Pi)_{j, k}, 
\end{aligned}
\]
that is, $2\pi i J = {}^t \Pi \cdot J \cdot \Pi$.
Inverting both sides of this equality and carrying out the explicit matrix computation we obtain
\[
H J^{-1} \, {}^t \Omega= 2\pi i I.
\]
In particular, equality of the coefficients in position $(j, j)$ yields
\begin{equation}\label{eq: Bertrand conclusion}
\sum_{m=1}^g \left(H_{j, c(m)} \Omega_{j,m} - H_{j, m} \Omega_{j, c(m)}\right) = 2\pi i.
\end{equation}
Note that the change of variables formula implies 
\[
\Omega_{j, m} = \int_{\lambda_m} \omega_j = \int_{\alpha_m \cdot \lambda} \omega_j = \int_\lambda \alpha_m^* (\omega_j) = \chi_j(\alpha_m) \int_\lambda \omega_j
\]
for all $1 \leq i,j \leq g$, and similarly 
\[
\Omega_{j, c(m)} %
= \chi_j\left(\frac{\beta_m}{c(\xi)}\right) \int_{\lambda} \omega_j, \quad H_{j, m} = \bar{\chi}_j(\alpha_m) \int_\lambda \omega_{c(j)}, \quad H_{j, c(m)} = \bar{\chi}_j\left( \frac{\beta_m}{c(\xi)} \right) \int_\lambda \omega_{c(j)}.
\]
Replacing these equalities into \Cref{eq: Bertrand conclusion}, we obtain
\begin{equation}
\label{eq: Bertand factors to be modified}
    \int_\lambda \omega_j \cdot \int_\lambda \omega_{c(j)} \cdot \sum_{m=1}^g \left( \bar{\chi}_j\left( \frac{\beta_m}{c(\xi)} \right)\chi_j(\alpha_m) - \bar{\chi}_j(\alpha_m)\chi_j\left( \frac{\beta_m}{c(\xi)}\right)  \right) = 2 \pi i.
\end{equation}
Observe that $\bar{\chi}_j(\alpha_m)=\chi_j(\bar{\alpha}_m) = \chi_j(\alpha_m)$ and similarly $\bar{\chi}_j(\beta_m) = \chi_j(\beta_m)$ since $\alpha_m, \beta_m$ lie in the totally real subfield $E_0$. The formula then simplifies to
\[
\int_\lambda \omega_j \cdot \int_\lambda \omega_{c(j)} \cdot \sum_{m=1}^g \left( \chi_j\left( \frac{\beta_m}{\xi} \right)\chi_j(\alpha_m) + \chi_j(\alpha_m)\chi_j\left( \frac{\beta_m}{\xi}\right)  \right) = 2 \pi i,
\]
or equivalently
\[
\int_\lambda \omega_j \cdot \int_\lambda \omega_{c(j)} \cdot \sum_{m=1}^g 2\chi_j(\alpha_m \beta_m)  = 2 \pi i \chi_j(\xi).
\]
Finally, we use the fact that $\sum_{m=1}^g \chi_j(\alpha_m\beta_m) = 1$ for all $j$: to see this, note that $\{\chi_1|_{E_0},\ldots,\chi_g|_{E_0}\}$ is the set of all embeddings of $E_0$ in $\C$ and form the matrices
\[
M = \left( \chi_i(\alpha_j) \right)_{i,j =1,\ldots,g} \quad \text{and} \quad N = \left( \chi_i(\beta_j) \right).
\]
The condition that $\{\beta_j\}_{j=1,\ldots,g}$ is the dual basis of $\{\alpha_j\}_{j=1,\ldots,g}$ with respect to the trace pairing yields ${}^t M \cdot N=I$. Hence we also have $N \cdot {}^t M = I$, which gives precisely $\sum_{m=1}^g \chi_j(\alpha_m\beta_m) = 1$ for all $j$. In conclusion, we obtain
\[
\int_\lambda \omega_j \cdot \int_\lambda \omega_{c(j)} = \pi i \chi_j(\xi),
\]
as desired.
\end{proof}

\subsection{Computational considerations}
\label{section: computing periods in the hard case}
The proof of \Cref{theorem: algebraic relation between period and quasiperiod} also gives the following more computationally-friendly result. Let $\lambda_1,\, \ldots,\, \lambda_{2g}$ be a basis of $\HH_{1}(A, \Q)$, symplectic with respect to a fixed polarization. Fix a basis $\omega_1,\, \ldots, \,\omega_{2g}$ of $\HH^1_{\dR}(A/kF)$ consisting of eigenforms, with corresponding characters $\chi_1,\, \ldots,\, \chi_{2g} : E \to \bar{k}^\times$. Suppose that $\omega_1,\, \ldots,\, \omega_g$ are holomorphic. Up to a normalization, we can assume that $\langle \omega_a^\vee, \omega_b^\vee\rangle_{\dR} = J_{ab}$, where $J$ is as in \Cref{eq: definition of matrix J}. Note that this normalization introduces a rescaling that is not immediate to compute: we describe how to circumvent this problem in \Cref{rmk: rescaling}.

Since $\HH_{1}(A, \Q)$ is a free $E$-module of rank $1$, there exist $e_1, \ldots, e_{2g} \in E$ such that $\lambda_j = e_j \lambda_1$. 

\begin{remark}
Suppose $\lambda_1,\, \ldots,\, \lambda_{2g}$ is the basis used for the matrices $A_\alpha$. Finding $e_j$ amounts to some straightforward linear algebra: if we identify $E$ with the $\Z$-linear span of the matrices $A_\alpha$, the element $e_j$ is simply the unique matrix in this span whose first column is the $j$-th vector of the standard basis of $\Q^{2g}$.
\end{remark}

Applying the argument in the proof of \Cref{theorem: algebraic relation between period and quasiperiod} to the chosen polarization, and replacing $\alpha_m$ with $e_m$ for $m=1, \dots, g$, and $\beta_m/c(\xi)$ with $e_{m+g}$ for $m=1, \dots, g$, we get
\begin{equation}
    \int_\lambda \omega_j \cdot \int_\lambda \omega_{c(j)} \cdot \sum_{m=1}^g \left( \bar{\chi}_j\left( e_{m+g} \right)\chi_j(e_m) - \bar{\chi}_j(e_m)\chi_j\left( e_{m+g}\right)  \right) = 2 \pi i.
\end{equation}
in lieu of \Cref{eq: Bertand factors to be modified}.
In particular, if $\int_{\lambda_1}\omega_j$ is known, then so is $\int_{\lambda_1} \omega_{j+g}$.
This is especially useful in practice: the algorithms currently implemented in \textsc{Magma} to compute period matrices of curves only return the integrals of a basis of holomorphic forms along $\lambda_1,\, \ldots,\, \lambda_{2g}$, whereas our application requires access also to periods of differentials of the second kind.

\begin{remark}\label{rmk: rescaling} 
We explain how to compute the normalization needed to ensure that the eigenbasis $\{\omega_j\}$ is symplectic. We assume that we have access to a basis $\{u_i\}$ of $\HH^1_{\dR}(A/k)$ which is symplectic for the fixed polarization, and that we can compute the action of $\iota E$ on $\HH^0(A, \Omega^1)$ with respect to this basis. Given the available algorithms, this assumption holds for Jacobians.

Note first that the data $\{u_j\},\, \{M_\alpha\}$ is clearly sufficient to compute an $E$-eigenbasis of $\HH^0(A_{kF}, \Omega^1)$. 
Next, we discuss the action of $\alpha \in E$ on $\HH^1_{\dR}(A/k)$.
Write $\alpha^* u_j = \sum_{i=1}^{2g} {d_{ij}} u_i$. We compute, for $1 \leq k \leq g$ and $g+1 \leq j \leq 2g$,
    \[
    \langle u_k, \alpha^* u_j \rangle = \sum_{i=1}^{2g} {d_{ij}} \langle u_k, u_i \rangle %
    = {d_{k+g, j} \, \langle u_k, u_{k+g}\rangle}
    \]
and
    \[
    \langle u_k, \alpha^* u_j \rangle = \langle (\alpha^\dagger)^*u_k, u_j \rangle = \sum_{m=1}^g (M_{\bar{\alpha}})_{mk} \langle u_m, u_j \rangle = (M_{\bar{\alpha}})_{j-g, k} \langle u_{j-g}, u_j\rangle.
    \]
Since $\langle u_k, u_{k+g}\rangle = \langle u_{j-g}, u_j\rangle$ for all $j,\, k$ in our range, comparing these formulas we obtain that the action of $\alpha$ on $\HH^1_\dR(X/k)$ is represented (with respect to the basis $\{u_i\}_i$) by a matrix of the form
    \[\begin{pmatrix}
        M_\alpha & Q \\
        0 & {}^tM_{\bar{\alpha}}
    \end{pmatrix}\] 
where $Q$ is some unknown $g \times g$ matrix.
    
Suppose now that $\omega_{c(j)} = \sum_{k=1}^{2g} f_{jk} u_k \in \HH^1_{\dR}(A/kF)$ is an eigenform with system of eigenvalues $\bar{\chi}_j$. A direct matrix computation shows that for every $\alpha \in E$ the vector ${}^{t}(f_{j, g+1}, \ldots, f_{j, 2g})$ is an eigenvector of ${}^tM_{\bar\alpha}$ of eigenvalue $\bar{\chi}_j(\alpha)$, and can therefore be computed (up to scalars) just from the knowledge of the matrices ${}^tM_{\bar{\alpha}}$, which we have access to.
Finally, observe that in order to apply the computational strategy described above we need to modify $\omega_{c(1)}, \ldots, \omega_{c(g)}$ so that $\omega_1, \ldots, \omega_{2g}$ is a symplectic basis. To do this, it suffices to rescale $\omega_{c(j)}$ by the inverse of $\langle \omega_j, \omega_{c(j)} \rangle$ (all other pairings are trivial, see \Cref{lemma: orthogonality differential forms}). Thus, it is enough to compute the pairing $\langle \omega_j, \omega_{c(j)} \rangle$ for every $j=1,\ldots,g$.
As $\HH^0(A, \Omega^1)$ is a totally isotropic subspace with respect to any polarization, we obtain
    \[
    \langle \omega_j, \omega_{c(j)}\rangle = \langle \omega_j, \sum_{k=1}^{2g} f_{jk} u_k \rangle = \langle \omega_j, \sum_{k=g+1}^{2g} f_{jk} u_k \rangle.
    \]
Since we have explicit expressions for the $\omega_j$ in terms of the symplectic basis $\{u_h\}$ and we know all the coefficients $f_{jk}$ for $k=g+1, \ldots, 2g$, the previous formula determines all the nontrivial pairings $\langle \omega_j, \omega_{c(j)}\rangle$, without the need to determine the coordinates $f_{j, 1},\, \dots, f_{j, g}$.
\end{remark}

\section{Examples}
\label{section: examples}
In this final section, we apply the preceding theory to compute the connected monodromy fields of two Jacobians.
Note that it is generally difficult to produce explicit examples in this setting. Suppose we look for a curve $X/k$ whose Jacobian $J$ is degenerate and has complex multiplication. Degeneracy can only occur when $\dim J = g(X) \geq 4$ \cite{MR608640}, and Coleman \cite[Conjecture~6]{Coleman} conjectured that, up to complex isomorphism, there exist only finitely many CM Jacobians of any fixed genus $g \geq 4$. While the conjecture is now known to fail for small genus ($g \leq 7$, see \cite{MR1085259, MR2510071}), it remains open in higher dimensions, and in any case CM Jacobians are relatively rare and hard to find. By contrast, constructing CM abelian varieties is comparatively straightforward, particularly if one allows the base field to vary. However, a generic CM abelian variety does not yield a degenerate example. Thus, although our method applies in principle to any CM abelian variety, the supply of genuinely interesting test cases among Jacobians is limited.

We now turn to our examples, referring to \cite{myrepo2025} for the supporting code.

\begin{example}[Fermat curve]
\label{example: Fermat Jacobian}
    Let $J/\Q$ be the Jacobian of the hyperelliptic curve $y^2=x^{15}+1$. {There is an action of the $15$-th roots of unity on $J$ that induces an embedding $\iota : \Q(\zeta_{15}) \hookrightarrow \End^0(J_{\bar{\Q}})$, so $J$ has CM.} The basis of hyperelliptic differentials $\omega_i = x^{i-1}\, dx/y \in \HH_{\dR}^1(J/\Q)$ diagonalizes the endomorphism action. The CM type of $J$ and a family $\mathcal{F}$ of equations for $\MT(J)$ are described in \cite[Example 4.2.9]{gallese_part1}.
    The endomorphism field $\Q(\End J)$ is $\Q(\zeta_{15})$.
    By \Cref{remark: if k is Q then Kconn is generated by periods}, \Cref{theorem: computation of kconn} simplifies, and $k(\varepsilon_J)$ is generated by the periods $P(\sigma, \omega_f)$.
    Of the nine equations in $\mathcal{F}$, only $f = x^{10}x^{12}/x^{8}x^{14}$ is expected to give a period $P(\sigma, \omega_f)$ not in $\Q(\zeta_{15})$.
    We compute the integrals of the differential forms $\omega_f$ via the method explained in \Cref{section: computing periods}. We get
    \[ P(\sigma, \omega_f) = 0.82990 \; 95361 \;82568... \]
    We need 500 digits of precision to recognize $P(\sigma, \omega_f)$ as an algebraic number and a generator for the number field {\texttt{\cite[\href{https://www.lmfdb.org/NumberField/16.0.3243658447265625.1}{number field 16.0.3243658447265625.1}]{lmfdb}}},
    confirming the results of \cite{gallese_part1}.%
\end{example}

\begin{example}[CM Jacobian of Mumford type]\label{ex: Mumford I}
Mumford \cite{MR248146} famously constructed examples of degenerate abelian fourfolds, parametrized by certain Shimura curves. Until recently, however, no explicit Jacobian of Mumford type was known. The situation has now changed with the work of Bouchet, Hanselman, Pieper, and Schiavone \cite{bouchet2025}, who computed the universal curve over some of Mumford's Shimura curves via interpolation through CM points in these moduli spaces. As a consequence of their work, we now have explicit examples of Mumford Jacobians without extra endomorphisms, as well as degenerate CM Jacobians in genus~4. In this example we analyze one such CM point on the Shimura curve arising from the arithmetic (2,3,7) group. This is the Jacobian $J$ of a hyperelliptic curve $X$ over the number field $k=\Q(\zeta_7)^+$. 
Letting $\mu = \zeta_7+\bar{\zeta_7}$, the curve $X$ is the fiber of the family $C_{7,t}$ corresponding to
\[ t = \frac{350588211676416\mu^2 - 1415944436534208\mu + 1160374047771995}{{601617706932059}}. \]
The defining equation of $X$ is too large to reproduce conveniently here, but is available in the accompanying GitHub repository \cite{myrepo2025}. %

    The endomorphism algebra of $J$ is isomorphic to the degree-8 CM field $E$ described in \texttt{\cite[\href{https://www.lmfdb.org/NumberField/8.0.7834003547041.1}{8.0.7834003547041.1}]{lmfdb}}.
    The Galois closure of $E/\Q$ is a field $F/\Q$ of degree 24.
    We verify that $J$ is simple via \cite[Chapter 1, Theorem 3.6]{Lang}.
    The reflex field $E^\ast$ has degree $2$ over $k$; it is the only number field of degree 6 that arises as a reflex field of $E$ (for varying CM type).
    From the simplicity of $J$, it follows that $k(\End J) = E^\ast$ \cite[Chapter 3, Theorem 1.1]{Lang}.

    The computation in this case is more involved than in \Cref{example: Fermat Jacobian}, for two reasons: $E/\Q$ is not Galois, and the hyperelliptic basis is not an eigenbasis for the endomorphism action.     
    We compute a symplectic eigenbasis of $\HH^1(X/kF)$ as in \Cref{section: computing periods in the hard case}, see especially \Cref{rmk: rescaling}.
    The Mumford-Tate group is defined by $4$ equations, $3$ of which come from the endomorphisms and the polarization, and an extra one in degree 2 corresponding to a cycle not generated by divisors.
    For each $f \in \mathcal{F}$, we compute the period $P(\sigma, \omega_f)$ by the method explained in \Cref{section: computing periods in the hard case}. We find that all these periods are in $F$.
    The subgroup of $\tau \in \Gal(F/k)$ such that $\tau P(\sigma, \omega_f)=P(\sigma,\omega_{\tau f})$ fixes the subfield $k(\varepsilon_J)=E^\ast$.
    In particular, although $J$ is degenerate, the equality $k(\varepsilon_J)=k(\End J)$ holds.
\end{example}

\begin{example}\label{ex: Mumford II}
    We consider a second CM Jacobian $J$ of Mumford type arising from the work of Bouchet, Hanselman, Pieper, and Schiavone \cite{bouchet2025}. This Jacobian is defined over $k=\Q(\zeta_9)^+$ and corresponds to a non-hyperelliptic curve $X/k$ of genus~4. The defining equations of $X$ are again too large to reproduce here, but they are available at \cite{myrepo2025}.
    Performing the computation as in the previous example, we find that the $k(\varepsilon_J)=k(\End J) = E^\ast$ is the field with label \texttt{\cite[\href{https://www.lmfdb.org/NumberField/6.0.465831.1}{6.0.465831.1}]{lmfdb}}.
\end{example}

\begin{example}
    Consider the curves
    \[ X\colon y^7=x(1-x), \qquad E\colon y^2+xy = x^3-x^2-2x-1. \]
    Denote by $J$ the Jacobian of $X$.
    For a non-zero square-free integer $d \neq 1$, let $E^{(d)}$ be the twist of $E$ by the non-trivial character of $\Q(\sqrt{d})$. Consider the abelian variety $A^{(d)} = J \times E^{(d)}$.
    Silverberg and Zarhin show in {\cite[Example 4.2]{MR1630512}} that
    \[
    \Q(\End(A)) = \Q(\End(J)) = \Q(\zeta_7),
    \qquad
    \Q(\varepsilon_A) = \Q(\zeta_7, \sqrt{d}).
    \]
    We have verified this result numerically for several values of $d$.  
    For instance, when $d=-1$, we obtain the nontrivial period
    \[ 1.36946\; 08463 \; 40899 \; \dots \]
    which we recognize as the square root of
    \[ (-1) \cdot  \left(\frac{12\zeta_7^5 + 24\zeta_7^4 + 36\zeta_7^3 + 48\zeta_7^2 + 60\zeta_7 + 
    30}{7}\right)^2. \]
\end{example}

\bibliographystyle{alpha}
\bibliography{biblio.bib}
\end{document}